\documentclass[11pt]{amsart}
\usepackage{amsfonts}
\usepackage{hyperref}
\usepackage[utf8]{inputenc}
\usepackage[T2A]{fontenc}
\usepackage[ukrainian,english]{babel}
\usepackage[dvips]{graphicx}
\usepackage{amssymb}
\usepackage{amsmath}
\usepackage{dsfont}
\usepackage{color}
\usepackage{latexsym}
\usepackage{bbm}
\usepackage{color}
\usepackage{amsthm}
\usepackage{multicol}
\usepackage[top   = 2.75cm,
bottom = 2.50cm,
left   = 2.50cm,
right  = 2.00cm]{geometry}

\bibstyle{plain}
\theoremstyle{plain}
\newtheorem{theorem}{Theorem}
\newtheorem{lemma}[theorem]{Lemma}

\newtheorem{corollary}[theorem]{Corollary}

\newtheorem{definition}[theorem]{Definition}

\theoremstyle{remark}

\newtheorem{remark}[theorem]{Remark}

\begin{document}

\title[]{Infinite continued fractions with a finite alphabet and their applications}
\author[M.\ V.\ Pratsiovytyi, S.\ P.\ Ratushniak]{ M.\ V.\ Pratsiovytyi, S.\ P.\ Ratushniak}

\newcommand{\eacr}{\newline\indent}

\address{M.V. Pratsiovytyi\eacr
Institute of Mathematics of NASU,
 Dragomanov Ukrainian State University, Kyiv,
Ukraine\acr ORCID 0000-0001-6130-9413}
\email{prats4444@gmail.com}

\address{S.P. Ratushniak\eacr
Institute of Mathematics of NASU,
 Dragomanov Ukrainian State University, Kyiv,
Ukraine\acr ORCID 0009-0005-2849-6233}
\email{ratush404@gmail.com}

\keywords{Сontinued fraction, $A_s$-representation of numbers, 
zero redundancy, number representation systems, cylinder sets, 
metric properties, topological properties, fractal analysis. \\This work was supported by a grant from the Simons
Foundation (SFI-PD-Ukraine-00014586, M.P., S.R.)}

\date{\today}

\newcommand{\acr}{\newline\indent}

\begin{abstract}
The paper investigates the topological and metric properties of the set $E$ of values of all infinite continued fractions of the form
\[0+\frac{1}{a_1+\dfrac{1}{a_2+_{\ddots}}}=1/a_1+1/a_2+...+1/a_n+...\equiv[0;a_1,a_2,...,a_n,...],\]
whose partial quotients ($a_n$) take values in a finite set of positive real numbers  $\{e_0,e_1,...,e_{s-1}\}$, $e_0<e_1<...<e_{s-1}$. It is established that the set $E$ is bounded, has cardinality continuum, and is perfect. Conditions under which $E$ is an interval, a nowhere dense set, or a set of Lebesgue measure zero are established. Necessary and sufficient conditions are obtained for $E$  to be an interval and for the corresponding coding system by such continued fractions to have zero redundancy, i.e., for each number to have at most two representations.

The geometric meaning of the digits in this representation, as well as the metric relations, is described in terms of the properties of cylindrical sets (cylinders and cylindrical intervals). It is proved that the basic metric ratio, defined as the ratio of the diameter of a cylinder to the diameter of its parent cylinder, is bounded away from both zero and one. This property is important for various applications of this representation in fractal analysis, the theory of continuous functions with locally complicated structure, and singular probability measures.
\end{abstract}

\maketitle

\section*{Introduction}
Continued fractions have numerous applications in various fields of mathematics and beyond. They play a particularly important role in number theory and approximation theory~\cite{Bodnar_Skorob,irwin,iosifescu,jonesthron,chin,pahir,PratsMakCh_2019}, as well as in the theory of singular functions and the theory of distributions of random variables~\cite{kul,PratsCh_2019,Lechinskii,pr,vynnyshyn1,vynnyshyn2}. Continued-fraction transformations can be considered in the broader context of dynamical systems and their discrete and continuous analogues~\cite{Stanzhytskyi}.

In~\cite{DKPr}, a model for representing a real number from a certain interval in the form of an infinite continued fraction (an $A_2$-continued fraction) was proposed. Its elements belong to the set
$A_2=\{e_0,e_1\}$
of positive real numbers satisfying the condition $e_0e_1=\frac{1}{2}$. This system for encoding numbers from an entire interval by means of a two-symbol alphabet has zero redundancy, i.e., each number has at most two formally distinct representations. Moreover, the numbers having two such representations (the $A_2$-binary numbers) form a countable everywhere dense subset of the interval, which is the set of values of such continued fractions.

Unlike the classical binary representation and some of its generalizations, the representation of numbers by $A_2$-continued fractions is not self-similar, which complicates the study of problems in metric number theory and fractal analysis. At the same time, the combination of the advantages of a two-symbol alphabet and the continued fraction structure provides, in many respects, certain relative theoretical convenience~\cite{Prats_monog}: the former provides technical convenience, while the latter ensures a higher rate of convergence. $A_2$-representations of numbers have been used in fractal analysis, metric number theory, and the theory of singular continuous measures~\cite{DKPr,prk,PratsKyurchev}, as well as in the theory of continuous functions with locally complicated behavior and structurally fractal properties~\cite{PGLR}, among other areas.

In the present paper, we consider infinite $A_s$-continued fractions ($s\geq 2$), which generalize $A_2$-continued fractions, and the corresponding systems for encoding real numbers, with both zero and nonzero redundancy, obtained by expanding numbers into $A_s$-continued fractions. We study the properties of $A_s$-representations of numbers and their geometry, including the geometric meaning of digits, properties of cylinder sets, and certain metric relations. We also investigate the topological and metric properties of the set of values of all $A_s$-continued fractions and establish, in particular, a criterion for the zero redundancy of the corresponding representation system.

 \section{Infinite continued $A_s$-fractions}

 Let $2\leq s$ be a fixed positive integer, and let $(e_0,e_1,...,e_{s-1})$ be a given set of positive numbers satisfying  $0<e_0<e_1<...<e_{s-1}$. The set $A_s\equiv\{e_0,e_1,...,e_{s-1}\}$ is called an alphabet, while the set $L_s\equiv A_s\times A_s\times... =\{(a_n):~ a_n\in A_s\}$ is called the space of sequences over the alphabet $A_s$.

An expression of the form
  \begin{equation}\label{fr}
  0+\frac{1}{a_1+\dfrac{1}{a_2+_{\ddots}}}=
  1/a_1+1/a_2+...+1/a_n+...\equiv[0;a_1,a_2,...,a_n,...],
  \end{equation}
  where$(a_n)\in L_s$, is called an \emph{infinite $A_s$-continued fraction} (briefly, an $A_s$-fraction). Since $a_1+a_2+...=\infty$ for every sequence $(a_n)\in L_s$, it follows from Seidel's theorem~\cite{seidel} that every $A_s$-fraction converges; that is, the sequence of its convergents has a finite limit $$\lim\limits_{n\to\infty}\frac{p_n(a_n)}{q_n(a_n)}=
  x=[0;a_1,a_2,...,a_n,...],$$
where  $p_{n+1}=a_{n+1}p_{n}+p_{n-1}$, $q_{n+1}=a_{n+1}q_{n}+q_{n-1}$, $p_0\equiv 0$, $p_1\equiv1$, $q_0\equiv 1$, $q_1\equiv a_1$.
 
The symbolic notation $[0;a_1,a_2,...,a_n,...]$ of the continued fraction~\eqref{fr} together with its value, is called its $A_s$-representation. Parentheses are used to indicate a period in the sequence of elements of the representation. We consider only infinite continued fractions.
 
Our main object of study is the set $E \equiv E(A_s)$ of values of all $A_s$-fractions, together with its structural, topological-metric, and fractal properties.
 \begin{theorem}\label{lem:porivn}
For every sequence of pairs $(a_n,b_n)\in A_s^2\times A_s^2\times\cdots$ the following double inequality holds:
\begin{equation}\label{eq0}
  [0;(e_{s-1},e_0)]\leq [0;a_1,b_1,\ldots,a_n,b_n,\ldots]\leq [0;(e_0,e_{s-1})].
\end{equation}
\begin{equation}\label{eq:1}
\min E=[0;(e_{s-1},e_0)]=\dfrac{\sqrt{e_0e_{s-1}(e_0e_{s-1}+4)}-e_0e_{s-1}}{2e_{s-1}}\equiv d_0,
\end{equation}
\begin{equation}\label{eq:2}
\max E=[0;(e_0, e_{s-1})]=\dfrac{\sqrt{e_0e_{s-1}(e_0e_{s-1}+4)}-e_0e_{s-1}}{2e_0}\equiv d_1.
\end{equation}
\end{theorem}
\begin{proof}
 First, we show that, for an arbitrary $n\in N$,
  \begin{equation}\label{eq0n}
     [0;\underbrace{e_{s-1},e_0,\ldots,e_{s-1},e_0}_{2n}]\leq [0;a_1,b_1,\ldots,a_n,b_n]\leq [0;\underbrace{e_0,e_{s-1},\ldots,e_0,e_{s-1}}_{2n}].
  \end{equation}
  To this end, we use the method of mathematical induction. Let $n=1$. In this case, the double inequality~\eqref{eq0n} takes the form
  \begin{equation}\label{eq0n1}
    \frac{1}{e_{s-1}+\frac{1}{e_0}}\leq\frac{1}{a_1+\frac{1}{b_1}}\leq\frac{1}{e_0+\frac{1}{e_{s-1}}}\Longleftrightarrow
    \frac{e_0}{e_0e_{s-1}+1}\leq\frac{b_1}{a_1b_1+1}\leq\frac{e_{s-1}}{e_0e_{s-1}+1}.
  \end{equation}
   Consider the difference
   \[\rho=\frac{b_1}{a_1b_1+1}-\frac{e_0}{e_0e_{s-1}+1}=\frac{e_0e_{s-1}b_1+b_1-e_0a_1b_1-e_0}{(a_1b_1+1)(e_0e_{s-1}+1)}=
   \frac{e_0b_1(e_{s-1}-a_1)+(b_1-e_0)}{(a_1b_1+1)(e_0e_{s-1}+1)}.\]
   Since $e_{s-1}-a_1\geq 0$ and $b_1-e_0\geq 0$, we have $\rho\geq0$. Hence, the left-hand inequality in~\eqref{eq0n1} holds. The right-hand inequality is proved analogously.
   
   Assume that inequality~\eqref{eq0n} holds for $n=k$. Consider $n=k+1$. Since
   \[[0;\underbrace{e_{s-1},e_0,\ldots,e_{s-1},e_0}_{2k+2}]=[0;e_{s-1},e_0+u], \mbox{ where }
   u=[0;\underbrace{e_{s-1},e_0,\ldots,e_{s-1},e_0}_{2k}],\]
    \[[0;a_1,b_1,\ldots,a_{k+1},b_{k+1}]=[0;a_1,b_1+c], \mbox{ where }
   c=[0;a_2,b_2,\ldots,a_{k+1},b_{k+1}],\]
     \[[0;\underbrace{e_0,e_{s-1},\ldots,e_0,e_{s-1}}_{2k+2}]=[0;e_{0},e_{s-1}+v], \mbox{ where }
   v=[0;\underbrace{e_0,e_{s-1},\ldots,e_{0},e_{s-1}}_{2k}],\]
   the induction hypothesis gives $u\leq c\leq v$. Therefore,
   \[[0;e_{s-1},e_0+u]\leq [0;a_1,b_1+c]\leq [0;e_0,e_{s-1}+v].\]
   By the principle of mathematical induction, inequality~\eqref{eq0n} holds for every $n$. Passing to the limit in~\eqref{eq0n} yields the double inequality~\eqref{eq0}. Therefore
     \[\min E=\min\limits_{(a_n)\in L_s}\{[0;a_1,a_2,\ldots,a_n,\ldots]\}=[0;(e_{s-1},e_0)]=d_0,\]
   \[\max E=\max\limits_{(a_n)\in L_s}\{[0;a_1,a_2,\ldots,a_n,\ldots]\}=[0;(e_{0},e_{s-1})]=d_1.\]
   The theorem is proved.
  \end{proof}
\begin{corollary}
The following equality holds:
\begin{equation}\label{eq:3}
\dfrac{d_0}{d_1}=\dfrac{e_0}{e_{s-1}} \Leftrightarrow \frac{d_0}{e_0}=\frac{d_1}{e_{s-1}},
\end{equation}
and, in particular, if $e_0e_{s-1}=\frac{1}{2}$ then $d_0=e_0$, $d_1=e_{s-1}$.
\end{corollary}
\begin{corollary}\label{cor2}
If $d_0\nabla e_0$, then $d_1\nabla e_{s-1}$, where $\nabla$ denotes an order relation $(=, <,>)$.
\end{corollary}
In what follows, we use the abbreviation $e_0e_{s-1}\equiv p$.
\begin{lemma}
1. If $e_0e_{s-1}=\frac{1}{2}$, then $d_0=e_0$ and $d_1=e_{s-1}$.

2. If $e_0e_{s-1}<\frac{1}{2}$, then $d_0>e_0$ and $d_1>e_{s-1}$.

3. If $e_0e_{s-1}>\frac{1}{2}$, then $d_0<e_0$ and $d_1<e_{s-1}$.
\end{lemma}
\begin{proof}
  In view of Corollary~\ref{cor2}, it suffices to establish the relation between $d_0$ and $e_0$.
 Since $p>0$, the equality $d_0=e_0$ is equivalent to $\sqrt{p(p+4)}-p=2p$, which is equivalent to $p=\frac{1}{2}$.

 The inequality $d_0<e_0$ is equivalent to $\sqrt{p(p+4)}-p<2p$, which is equivalent to $p>\frac{1}{2}$.

 The inequality $d_0>e_{0}$ is equivalent to $\sqrt{p(p+4)}>3p$, which is equivalent to $p<\frac{1}{2}$.
\end{proof}
\begin{lemma} For the endpoints $d_0$ and $d_1$ of the set $E$ of values of $A_s$-fractions, the following relation holds:
\begin{equation}\label{dobk}
  d_0d_1=\frac{1}{2}(p+2-\sqrt{p^2+4p})<1.
\end{equation}
\end{lemma}
\begin{proof}
 Indeed, from
  \begin{align*}
    \varphi(p)\equiv&d_0d_1=\frac{(\sqrt{p(p+4)}-p)^2}{4p}=\frac{2p(p+2-\sqrt{p^2+4p})}{4p}= \\
    =&\frac{1}{2}(p+2-\sqrt{p^2+4p})=\frac{p+2}{2}\left(1-\frac{\sqrt{(p+2)^2-4}}{p+2}\right)=\\
    =&\frac{p+2}{2}\left(1-\sqrt{1-\frac{4}{(p+2)^2}}\right),
  \end{align*}
  we see that the function $\varphi(p)=\frac{1}{2}(p+2-\sqrt{p^2+4p})$ is decreasing on $[0;\infty)$. Therefore, since $p>0$ we have $\varphi(p)<\varphi(0)=1$, and hence $d_0d_1<1$. The lemma is proved.
\end{proof}
\begin{lemma} The following inequality holds: $(e_0+d_0)(e_0+d_1)> 1.$
\end{lemma}
\begin{proof} We transform the left-hand side as follows:
 \begin{align*}
  (e_0+d_0)(e_0+d_1)&=(e_0+\frac{\sqrt{p(p+4)}-p}{2e_{s-1}})(e_0+\frac{\sqrt{p(p+4)}-p}{2e_{0}})=\\
  =&e_0^2+\frac{e_0(\sqrt{p(p+4)}-p)}{2e_{s-1}}+\frac{e_0(\sqrt{p(p+4)}-p)}{2e_{0}}+\frac{1}{2}(p+2-\sqrt{p(p+4)})=\\
  =&e_0^2+\frac{e_0}{2e_{s-1}}\sqrt{p(p+4)}-\frac{e^2_0}{2}+\frac{\sqrt{p(p+4)}}{2}-\frac{p}{2}+\frac{p}{2}+1-\frac{\sqrt{p(p+4)}}{2}= \\
  =&\frac{e^2_0}{2p}(p+\sqrt{p(p+4)})+1> 1. \qedhere
\end{align*}
\end{proof}
\begin{definition}
Let $(c_1,c_2,\ldots,c_m)$ be a fixed finite sequence of elements of the alphabet.
The set
\[
\Delta^{A_s}_{c_1c_2\ldots c_m}
=
\left\{
x:\ x=[0;c_1,c_2,\ldots,c_m,\alpha_1,\alpha_2,\ldots],
\quad (\alpha_n)\in L_s
\right\}
\]
of numbers from the interval $[d_0;d_1]$ is called a \emph{cylinder of rank $m$ with base $c_1c_2\ldots c_m$}.
\end{definition}
Obviously, $\Delta^{A_s}_{c_1c_2\ldots c_mc}
\subset
\Delta^{A_s}_{c_1c_2\ldots c_m}$.
Moreover, $E=\bigcup\limits_{i=0}^{s-1}\Delta_{e_i}^{A_s}$
and, in general,
\[
E=
\bigcup\limits_{c_1=0}^{s-1}\cdots
\bigcup\limits_{c_m=0}^{s-1}
\Delta^{A_s}_{c_1\ldots c_m}.
\]
\begin{definition} The minimal interval containing the cylinder $\Delta^{A_s}_{c_1c_2...c_m}$ is denoted by $\bar{\Delta}^{A_s}_{c_1...c_m}$ and is called the cylindrical interval of rank $m$ with base $c_1c_2...c_m$.
\end{definition} 
Since we consider only infinite continued fractions, it is readily seen that the endpoints of the cylindrical interval $\bar{\Delta}^{A_s}_{c_1...c_m}$, as well as those of the corresponding cylinder, are the points
$[0;c_1,...,c_{m},(e_{s-1},e_0)]=[0;c_1,...,c_m,d_0^{-1}]$
and $[0;c_1,...,c_m,(e_0,e_{s-1})]=[0;c_1,...,c_m,d_1^{-1}]$.

For even ranks, the first point is the left endpoint, whereas for odd ranks, it is the right endpoint.
\begin{lemma}
The length of the cylindrical interval
$\bar{\Delta}^{A_s}_{c_1...c_m}$, which is the diameter of the corresponding cylinder, is given by
\[|\bar{\Delta}^{A_s}_{c_1...c_m}|=
  \frac{d_1-d_0}{(q_m+d_1q_{m-1})(q_m+d_0q_{m-1})},\]
  where $q_m$ is the denominator of the finite continued fraction $[0;c_1,c_2,...,c_m]=\frac{p_m}{q_m}.$
\end{lemma}
\begin{proof}
By definition, the length of the cylindrical interval (the diameter of the cylinder) is
\begin{align*}
  |\bar{\Delta}^{A_s}_{c_1...c_m}|&=\max\bar{\Delta}^{A_s}_{c_1...c_m}-\min\bar{\Delta}^{A_s}_{c_1...c_m}=\\
  &=|[0;a_1,...,a_m,d_1^{-1}]-[0;a_1,...,a_m,d_0^{-1}]|=\\
  &=|\frac{d_1^{-1}p_m+p_{m-1}}{d_1^{-1}q_m+q_{m-1}}-\frac{d_0^{-1}p_m+p_{m-1}}{d_0^{-1}q_m+q_{m-1}}|=\\
  &=\frac{|(p_mq_{m-1}-p_{m-1}q_m)(d_1-d_0)|}{(q_m+d_1q_{m-1})(q_m+d_0q_{m-1})}=\\
  &=\frac{d_1-d_0}{(q_m+d_1q_{m-1})(q_m+d_0q_{m-1})},
  \end{align*}
since $q_mp_{m-1}-p_mq_{m-1}=(-1)^m$, where $p_m$ is the numerator of the $m$-th convergent.
\end{proof}
\begin{corollary}
   The following estimates hold:
  \[|\bar{\Delta}^{A_s}_{c_1...c_m}|<
  \frac{d_1-d_0}{(q_m+d_0q_{m-1})^2}<\frac{d_1-d_0}{q_m^2}\to 0 (m\to\infty).\]
\end{corollary}
\begin{corollary}
  The main metric relation for the $A_s$-representation is given by
 \[  \frac{|\bar{\Delta}^{A_s}_{c_1...c_mc}|}{|\bar{\Delta}^{A_s}_{c_1...c_m}|}=
  \frac{(q_m+d_1q_{m-1})(q_m+d_0q_{m-1})}{((c+d_1)q_m+q_{m-1})((c+d_0)q_m+q_{m-1})}=
  \frac{(1+d_1\frac{q_{m-1}}{q_m})(1+d_0\frac{q_{m-1}}{q_m})}{(c+d_1+\frac{q_{m-1}}{q_m})(c+d_0+\frac{q_{m-1}}{q_m})}.\]
\end{corollary}
\begin{theorem}\label{th2}
For every $(c_1,...,c_m)\in A_s^m$ and $c\in A_s$ there exist constants $C_1$ and $C_2$ such that
\[0<C_1<\frac{|\bar{\Delta}^{A_s}_{c_1...c_mc}|}
{|\bar{\Delta}^{A_s}_{c_1...c_m}|}<C_2<1.\]
\end{theorem}
\begin{proof}
Since
$\frac{q_{m-1}}{q_m}=[0;a_m,a_{m-1},\ldots,a_1]
\in[\frac{e_0}{p+1};\frac{1}{e_0}]$
we have
\begin{equation}\label{c1}
      \frac{|\bar{\Delta}^{A_s}_{c_1...c_mc}|}
{|\bar{\Delta}^{A_s}_{c_1...c_m}|}>
  \left(\frac{1+d_0\frac{q_{m-1}}{q_m}}{c+d_1+\frac{q_{m-1}}
  {q_m}}\right)^2>
  \left(\frac{1+\frac{d_0e_0}{e_0e_{s-1}+1}}{e_{s-1}+d_1+\frac{1}
  {e_0}}\right)^2=C_1>0.
  \end{equation}
It remains to show that the fundamental metric ratio is bounded away from $1$. Clearly,
\[\frac{|\bar{\Delta}^{A_s}_{c_1...c_mc}|}
{|\bar{\Delta}^{A_s}_{c_1...c_m}|}=\frac{(1+d_1\frac{q_{m-1}}{q_m})(1+d_0\frac{q_{m-1}}{q_m})}{(c+d_1+\frac{q_{m-1}}{q_m})(c+d_0+\frac{q_{m-1}}{q_m})}<
\frac{(1+d_1\frac{q_{m-1}}{q_m})
(1+d_0\frac{q_{m-1}}{q_m})}
{(e_0+d_1+\frac{q_{m-1}}{q_m})
(e_0+d_0+\frac{q_{m-1}}{q_m})}.\]
Since
$\frac{q_{m-1}}{q_m}\in[\frac{e_0}{p+1};\frac{1}{e_0}]$
consider the function
$\gamma(x)=\frac{(1+d_1x)(1+d_0x)}{(e_0+d_1+x)(e_0+d_0+x)}$
defined on the interval
$[\frac{e_0}{p+1};\frac{1}{e_0}]$
Using inequality~\eqref{dobk}, we obtain
\begin{align*}
     \gamma(x)=& \frac{x^2d_0d_1+x(d_0+d_1)+1}{x^2+x(2e_0+d_0+d_1)+(e_0+d_0)(e_0+d_1)}<\\
  <& \frac{x^2+x(d_0+d_1)+1}{x^2+x(2e_0+d_0+d_1)+(e_0+d_0)(e_0+d_1)}=\\
  =&\frac{x^2+x(2e_0+d_0+d_1)+(e_0+d_0)(e_0+d_1)-2e_0x-(e_0+d_0)(e_0+d_1)+1}{x^2+x(2e_0+d_0+d_1)+(e_0+d_0)(e_0+d_1)}=\\
  =&1-\frac{2e_0x+(e_0+d_0)(e_0+d_1)-1}{x^2+x(2e_0+d_0+d_1)+(e_0+d_0)(e_0+d_1)}=\\
  =&1-\frac{2e_0x+\frac{e^2_0}{2p}(p+\sqrt{p(p+4)})+1-1}{x^2+x(2e_0+d_0+d_1\frac{e^2_0}{2p}(p+\sqrt{p(p+4)})+1}=\\
  =&1-\frac{2e_0x+\frac{e^2_0}{2p}(p+\sqrt{p(p+4)})}{x^2+x(2e_0+d_0+d_1)+\frac{e^2_0}{2p}(p+\sqrt{p(p+4)})+1}.
\end{align*}
By the preceding lemma,
$(e_0+d_0)(e_0+d_1)
=\frac{e_0^2}{2p}
\left(p+\sqrt{p(p+4)}\right)+1.$
Hence, putting
\[
r(x)=\frac{2e_0x+\frac{e^2_0}{2p}
(p+\sqrt{p(p+4)})}{x^2+x(2e_0+d_0+d_1)+
\frac{e^2_0}{2p}(p+\sqrt{p(p+4)})+1},
\]
we have
\[
\gamma(x)<1-r(x).
\]
The function $r(x)$ is continuous and strictly positive on 
$[\frac{e_0}{p+1},\frac{1}{e_0}]$. Therefore, by the Weierstrass extreme value theorem, it attains its minimum
\[m_0\equiv\min\limits_{x\in [\frac{e_0}{p+1},\frac{1}{e_0}]}r(x).\]
Since
$r(x)>0$ on this interval, we have $m_0>0$.
Consequently,
\[
\frac{|\bar{\Delta}^{A_s}_{c_1...c_mc}|}
{|\bar{\Delta}^{A_s}_{c_1...c_m}|}<\gamma(x)<1-m_0
=C_2<1.\qedhere\]
\end{proof}
\begin{corollary}
For any sequence $(c_n)\in L_s$, the following equality holds:
$$\bigcap\limits_{n=1}^{\infty}
\overline{\Delta}^{A_s}_{c_1c_2...c_n}\equiv\Delta^{A_s}_{c_1c_2...c_n...}=\bigcap\limits_{n=1}^{\infty}\Delta^{A_s}_{c_1...c_n}.$$

Thus, every point of the set $E$ can be regarded as a ``cylinder of infinite rank'', and we obtain another representation of the number
$\Delta^{A_s}_{c_1c_2...c_n...}$.
\end{corollary}
\begin{theorem}\label{th1}
The set $E$ of values of all infinite $A_s$-continued fractions is bounded, perfect, and has the cardinality of the continuum.
\end{theorem}
\begin{proof}
The boundedness of the set $E$ was established above:
$E\subseteq [d_0,d_1].$

Consider the sequence of sets $E_m$, where
\[E_m\equiv\bigcup\limits_{c_1\in A_s}\bigcup\limits_{c_2\in A_s}\ldots\bigcup\limits_{c_m\in A_s}\overline{\Delta}^{A_s}_{c_1c_2\ldots c_m}.\]
Each set $E_m$ is a finite union of closed intervals and is therefore perfect and bounded. Moreover $E_{m+1}\subset E_m$, $m\in N$.

Then
\[E=\bigcap\limits_{m=1}^{\infty}\left[\bigcup\limits_{c_1\in A_s}\bigcup\limits_{c_2\in A_s}\ldots\bigcup\limits_{c_m\in A_s}\overline{\Delta}^{A_s}_{c_1c_2\ldots c_m}\right]=\bigcap\limits_{m=1}^{\infty}E_m.\]

According to the theorem on the structure of closed sets, the set $E$ is closed.

We now show that $E$ has no isolated points.
Let $x_0$ be an arbitrary point of $E$ and let
$x_0=\Delta^{A_s}_{\alpha_1\ldots\alpha_n\ldots}$.
Then, for every $n\in\mathbb{N}$, $x_0$ belongs to the
cylindrical interval
$\overline{\Delta}^{A_s}_{\alpha_1\ldots\alpha_n}$
and is either an interior point or one of its endpoints:
$
x_n=\Delta^{A_s}_{\alpha_1\ldots\alpha_n(e_{s-1}e_0)}$ and
$x'_n=\Delta^{A_s}_{\alpha_1\ldots\alpha_n(e_0e_{s-1})}.$
Consider the sequence
\[
u_n=
\begin{cases}
x_n, & \text{if } x_n\ne x_0,\\
x'_n, & \text{if } x'_n\ne x_0,
\end{cases}
\qquad n\in\mathbb{N}.
\]
Clearly, $u_n\in E$ for every $n$. Since
\[
x_0\in\Delta^{A_s}_{\alpha_1\ldots\alpha_n}
\subseteq
\overline{\Delta}^{A_s}_{\alpha_1\ldots\alpha_n}
\]
and
\[
u_n\in\Delta^{A_s}_{\alpha_1\ldots\alpha_n}
\subseteq
\overline{\Delta}^{A_s}_{\alpha_1\ldots\alpha_n},
\]
while
\[
|
\overline{\Delta}^{A_s}_{\alpha_1\ldots\alpha_n}
|
=
d(\Delta^{A_s}_{\alpha_1\ldots\alpha_n})
\longrightarrow 0
\qquad (n\to\infty),
\]
we obtain
\[
\lim_{n\to\infty}u_n=x_0.
\]
Moreover, by construction, $u_n\ne x_0$. Hence, $x_0$ is an
accumulation point of $E$. Since $x_0$ was arbitrary, $E$ has no
isolated points.
 Together with the closedness of $E$ this implies that $E$ is a perfect set.

Since every nonempty perfect subset of $\mathbb R$ has cardinality continuum, the set $E$ has cardinality continuum.
\end{proof}
 \begin{lemma}\label{lemA2}
If, for some $i\in\{1,...,s-1\}$  $$e_{i}-e_{i-1}>d_1-d_0,$$ then the interval $(\frac{1}{e_i+d_0};\frac{1}{e_{i-1}+d_1})\subset[d_0;d_1]$ is a gap of $E$.
\end{lemma}
\begin{proof}
  Since the first-rank cylinders occur in the reverse order of their bases, we have
  \[\min\Delta^{A_s}_{e_i}<\min\Delta^{A_s}_{e_{i-1}}.\]
   The condition $e_{i}-e_{i-1}>d_1-d_0$ is equivalent to $e_i+d_0>e_{i-1}+d_1$ and hence 
   $\frac{1}{e_i+d_0}<\frac{1}{e_{i-1}+d_1}.$
      
      Therefore, the interval $(\frac{1}{e_i+d_0},\frac{1}{e_{i-1}+d_1})$
   lies strictly between the two adjacent first-rank cylinders
$\Delta^{A_s}_{e_i}$ and $\Delta^{A_s}_{e_{i-1}}$. Since
$E=\bigcup\limits_{k=0}^{s-1}\Delta^{A_s}_{e_k}$ we obtain $E\cap(\frac{1}{e_i+d_0};\frac{1}{e_{i-1}+d_1})=\emptyset$.
Moreover,
$\frac{1}{e_i+d_0}
=\max\Delta^{A_s}_{e_i}\in E$,
$\frac{1}{e_{i-1}+d_1}
=\min\Delta^{A_s}_{e_{i-1}}\in E.$
Thus, the interval is a gap of $E$.
\end{proof}
\begin{lemma}\label{lem6}
If, for all
 $i\in \{1,2,...,s-1\}$ the relation
  \begin{equation}\label{eq4}
    e_{i}-e_{i-1}\bigtriangledown d_1-d_0,
  \end{equation} holds, where $\bigtriangledown$ denotes one of the relations $\{=,<,>\}$, then $e_0e_{s-1} \bigtriangledown \frac{(s-1)^2}{s}.$
\end{lemma}
\begin{proof}
  Summing the relations~\eqref{eq4}, we obtain $e_{s-1}-e_0 \bigtriangledown (d_1-d_0)(s-1). $ Since $d_0=\frac{{e_0}{d_1}}{e_{s-1}},$ we have $$e_{s-1}-e_0 \bigtriangledown (1-\frac{e_0}{e_{s-1}})(s-1)d_1.$$ Hence
  $e_{s-1}\bigtriangledown (s-1)d_1.$ Putting ${e_0}{e_{s-1}}\equiv p$ and using the expression for $d_1$, we obtain
  \[2p\bigtriangledown(s-1)[\sqrt{p(p+4)}-p].\]
  Therefore, $(s+1)\bigtriangledown (s-1)\sqrt{p(p+4)}$ and $e_0e_{s-1}=p\bigtriangledown \frac{(s-1)^2}{s}$.
\end{proof}
\begin{corollary}
  If $s=2$, then the condition $e_0e_1\bigtriangledown\frac{1}{2}$ is equivalent to $e_1-e_0\bigtriangledown d_1-d_0$.
\end{corollary}
\begin{remark}
  It is clear that, for $s>2$ the equality $e_0e_{s-1} = \frac{(s-1)^2}{s}$ does not imply that $e_i-e_{i-1}=d_1-d_0$, $i=\overline{1,s-1}$.
\end{remark}
\begin{theorem}
   If $e_{i}-e_{i-1}>d_1-d_0$ for all $i\in \{1,2,...,s-1\}$, then the set $E$ is nowhere dense and has Lebesgue measure zero.
\end{theorem}
\begin{proof} By Lemma~\ref{lem6}, the assumptions of the theorem imply that the first-rank cylindrical intervals are pairwise disjoint, and the a complementary intervals between them contain no points of $E$. Similarly, if $\Delta^{A_s}_{e_i}$ is an arbitrary first-rank cylindrical interval, then the second-rank cylindrical intervals contained in it are pairwise disjoint, and the complementary intervals between them contain no points of $E$. The same holds for cylindrical intervals of all higher ranks. Therefore
\[E=\bigcap\limits_{m=1}^{\infty}\left[\bigcup\limits_{c_1=0}^{s-1}...\bigcup\limits_{c_m=0}^{s-1}\Delta^{A_s}_{c_1...c_m}\right],\]
 is perfect and nowhere dense.

It remains to prove that $E$ has Lebesgue measure zero.

Let $E_0=[d_0;d_1]$, $E_1=\bigcup\limits_{c_1=0}^{s-1}
\overline{\Delta}^{A_s}_{c_1}$,
$E_2=\bigcup\limits_{c_1=0}^{s-1}
\bigcup\limits_{c_2=0}^{s-1}
\overline{\Delta}^{A_s}_{c_1c_2},\ldots,
E_m=\bigcup\limits_{c_1=0}^{s-1}\cdots
\bigcup\limits_{c_m=0}^{s-1}\overline{\Delta}^{A_s}_{c_1\ldots c_m},\ldots$
Then
$E\subset E_m\subset E_{m-1}\subset\ldots\subset E_2\subset E_1\subset E_0$ and
\[E=\bigcap\limits_{m=1}^{\infty}E_m=
\lim\limits_{m\to\infty}E_m.\]
Hence, under the assumptions of the theorem, $\lambda(E)\leq \lambda(E_m)$ $\forall m\in N$ and
\[\lambda(E)=\lim\limits_{m\to\infty}\lambda(E_m).\]
By the properties of Lebesgue measure, we have
\[\lambda(E_m)=(d_1-d_0)
\frac{\lambda(E_m)}{\lambda(E_{m-1})}\cdot
\frac{\lambda(E_{m-1})}{\lambda(E_{m-2})}\cdot\cdots\cdot
\frac{\lambda(E_2)}{\lambda(E_{1})}\cdot
\frac{\lambda(E_1)}{\lambda(E_{0})},\]
and therefore
\[\lambda(E)=\lim\limits_{m\to\infty}\lambda(E_m)=
(d_1-d_0)\prod\limits_{m=1}^{\infty}\frac{\lambda(E_m)}{\lambda(E_{m-1})}.\]

Taking into account the structure of the sets $E_{m-1}$, namely $E_{m-1}=E_{m}\cup \overline{E}_{m}$, we have
$E_m=E_{m-1}\setminus \overline{E}_m$.
Consequently,
\[\lambda(E)=(d_1-d_0)\prod\limits_{m=1}^{\infty}
\frac{\lambda(E_{m-1})-\lambda(\overline{E}_m)}{\lambda(E_{m-1})}=
(d_1-d_0)\prod\limits_{m=1}^{\infty}
\left[1-\frac{\lambda(\overline{E}_m)}
{\lambda(E_{m-1})}\right].\]
Since $\overline{E}_m\subset E_{m-1}$, we have
$\frac{\lambda(\overline{E}_m)}
{\lambda(E_{m-1})}<1$. On the other hand, taking into account that the main metric ratio is bounded away from both $0$ and $1$, i.e., $0<C_1<\frac{\lambda({E}_m)}
{\lambda(E_{m-1})}<C_2<1$, we obtain
\[\frac{\lambda(E_m)}{\lambda(E_{m-1})}+\frac{\lambda(\overline{E}_m)}
{\lambda(E_{m-1})}=1,\]
and hence
\[\frac{\lambda(\overline{E}_m)}{\lambda(E_{m-1})}=1-
\frac{\lambda({E}_m)}
{\lambda(E_{m-1})}\leq1-C_1<1.\]
Thus, $0<\frac{\lambda(\overline{E}_m)}{\lambda(E_{m-1})}<1$
and $\sum\limits_{m=1}^{\infty}
\frac{\lambda(\overline{E}_m)}{\lambda(E_{m-1})}=\infty$.
Taking into account the relation between divergent infinite products and the corresponding series~\cite{Prats_monog}, we conclude that
$\prod\limits_{m=1}^{\infty}
\left[1-\frac{\lambda(\overline{E}_m)}
{\lambda(E_{m-1})}\right]$ converges to zero. Therefore,
$\lambda(E)=0$.
\end{proof}
\begin{corollary}\label{cor:1}
  If $s=2$ and $e_1-e_0>d_1-d_0$ (which is equivalent to $e_0e_{1}>\frac{1}{2}$), then $E$ є is a nowhere dense set of Lebesgue measure zero.
\end{corollary}
\begin{lemma}
  If there exists an integer $j\in\{1,2,...,s-1\}$, such that $e_{i+j}e_{i}\leq\frac{1}{2}$, then, for a fixed finite sequence of alphabet elements $(c_1,c_2,\ldots, c_m)$ and $i\in \{0,...,s-j-1\}$, the set
  \[\Delta'_{c_1...c_m}=\{x:~x=\Delta^{A_s}_{c_1...c_m\alpha_1\alpha_2...}, \alpha_n\in \{e_i,e_{i+j}\}\forall n\in N\}\]
is an interval entirely contained in $E$.
\end{lemma}
\begin{proof}
  Suppose that, for some $j\in \{1,2,\cdots,s-1\}$ and a fixed $i\in \{0,...,s-j-1\}$ $$e_{i+j}e_{i}\leq\frac{1}{2}.$$ Consider the set $E(A_2)$ of all values of $A_2$-fractions~\cite{DKPr} with alphabet $A_2=\{e_{i},e_{i+j}\}$. It was proved in~\cite{DKPr} that if the elements of the alphabet satisfy $e_{i+j}e_{i}\leq\frac{1}{2}$, then the set $E(A_2)$ is the interval $[\frac{1}{2e_{i+j}},\frac{1}{2e_i}]$. Since
  $A_2\subset A_s$, every $A_2$-fraction is clearly an $A_s$-fraction. Hence,
  $E(A_2)\subset E$, which implies that the set $\Delta'_{c_1...c_m}$ is an interval entirely contained in $E$.
\end{proof}
  \section{The $A_s$-continued fraction representation of real numbers}
We call the representation of a number $x\in E$
by an $A_s$-continued fraction its $A_s$-representation, while the symbolic notation $x=[0;a_1,a_2,...,a_n,...]$ is called the $A_s$-expansion of the number $x$.
  
\begin{theorem}\label{mth}
    The set $E$  of all values of infinite $A_s$-continued fractions is the
interval $[d_0,d_1]$, if and only if
    $e_i-e_{i-1}\leq d_1-d_0$  $\forall i\in \{1,2,...,s-1\}$ (in this case, $e_0e_{s-1}\leq \frac{(s-1)^2}{s}$).
  \end{theorem}
  \begin{proof}
    Since $E$ is closed (Theorem~\ref{th1}), the equality $E = [d_0;d_1]$ holds if and only if $E$ has no complementary
intervals. By Lemma~\ref{lemA2}, this is possible if and only if $\max{\Delta_{e_i}}\geq\min{\Delta_{e_{i-1}}},$ $i\in \{1,2,...,s-1\}$.
Equivalently,
    \[\frac{1}{e_i+d_0}\geq\frac{1}{e_{i-1}+d_1},\]
 which is equivalent to   \[
    e_{i-1}+d_1\geq e_{i}+d_0, \]
    and hence 
    \[e_{i}-e_{i-1}\leq d_1-d_0.\]
  Summing these inequalities for
  $i=1,\ldots,s-1$
\begin{align*}
 &e_1-e_0\leq d_1-d_0, \\
 &e_2-e_1\leq d_1-d_0, \\
 &\cdots\cdots\cdots\cdots\cdots, \\
 &e_{s-1}-e_{s-2}\leq d_1-d_0,
\end{align*}
we obtain
 $e_{s-1}-e_0\leq (s-1)(d_1-d_0).$ Since $\frac{d_0}{d_1}=\frac{e_0}{e_{s-1}}$, we have $d_0=\frac{e_0}{e_{s-1}}d_1$. Therefore,
 \[e_{s-1}-e_0\leq (s-1)d_1(1-\frac{e_0}{e_{s-1}}),\]
 \[1\leq \frac{(s-1)d_1}{e_{s-1}}=(s-1)\frac{d_0}{d_1}\Leftrightarrow
 e_{s-1}\leq (s-1)\cdot\frac{\sqrt{p(p+4)}-p}{2e_0}.\]
 Hence
 \[2p\leq (s-1)[\sqrt{p(p+4)}-p],\]
 \[2p\leq (s-1)\sqrt{p(p+4)}-p(s-1).\]
 Thus,
 $p\leq \frac{(s-1)^2}{s}$, which is equivalent to  $e_0e_{s-1}\leq\frac{(s-1)^2}{s}$.
   \end{proof}
Under the assumptions of the theorem, every number in the interval $[d_0,d_1]$ has an $A_s$-expansion.

\section{Zero-Redundancy Number Representation Systems}
By the coding of the elements of a set $X$ by means of an alphabet $A$, we mean a surjective mapping
$
X\to A\times A\times\cdots.$
A coding (representation) system is said to have zero redundancy if each element has at most two representations, and the set of elements having two representations is at most countable.

  \begin{theorem}\label{repreth}
    The set of values of all $A_s$-continued fractions is an interval, and
the coding system of numbers by means of $A_s$-continued fractions has
zero redundancy if and only if the numbers $e_0,e_1,...,e_{s-1}$ are consecutive terms of an arithmetic
progression with common difference $d\equiv d_1-d_0$ (в цьому випадку  $e_0e_{s-1}=\frac{(s-1)^2}{s}$).
  \end{theorem}
\begin{proof}
By Theorem~\ref{mth} and Lemma~\ref{lem6}, the equality $E=[d_0;d_1]$, together with zero redundancy of the coding system is
equivalent to the condition that the first-rank cylindrical intervals
have pairwise disjoint interiors and have common endpoints. 

This is possible if and only if, for every $i\in \{1,2,...,s-1\}$ 
\[\max{\Delta^{A_s}_{e_i}}=\min{\Delta^{A_s}_{e_{i-1}}}\Longleftrightarrow
e_{i}+d_0=e_{i-1}+d_1.\]
Since for all $i\in \{1,2,...,s-1\}$ we obtain
\begin{equation}\label{req}
  e_i-e_{i-1}=d_1-d_0=d.
\end{equation}
Thus, $e_0,e_1,...,e_{s-1}$ are consecutive terms of an arithmetic
progression. Summing the equalities~\eqref{req} above, we obtain
\begin{equation}\label{req1}
  e_{s-1}-e_0=(s-1)(d_1-d_0).
  \end{equation}
 Since $\frac{d_0}{d_1}=\frac{e_0}{e_{s-1}}$ we have $d_0=\frac{e_0}{e_{s-1}}d_1$. Substituting this into~\eqref{req1}, we obtain
  $$e_{s-1}-e_0=(s-1)d_1(1-\frac{e_0}{e_{s-1}}).$$
  Since $e_{s-1}-e_0$, cancellation yields
  $e_{s-1}=(s-1)d_1.$
  Using $d_1$, we obtain
  \[e_{s-1}=(s-1)\frac{\sqrt{e_0e_{s-1}(e_0e_{s-1}+4)}-e_0e_{s-1}}{2e_0},\]
  \[2e_0e_{s-1}=(s-1)[\sqrt{e_0e_{s-1}(e_0e_{s-1}+4)}-e_0e_{s-1}].\]
  Putting ${e_0}{e_{s-1}}\equiv p$ and solving the resulting equation, we obtain $e_0e_{s-1}=p=\frac{(s-1)^2}{s}.$
  \end{proof}
        \begin{corollary}
      For the coding system of the numbers in the interval $[d_0,d_1]$ by
means of $A_2$-continued fractions to have zero redundancy, it is
necessary and sufficient that$
e_0e_1=\frac{1}{2}.$
In this case,
$d_0=\frac{1}{2e_1}$, $d_1=\frac{1}{2e_0}$.
In particular, if $e_0=\frac{1}{2}$, then
$e_1=1$, $d_0=\frac{1}{2}$, $d_1=1$.
  \end{corollary}

  It is worth noting that the case $s=2$ deserves special attention due
to the minimal cardinality of the alphabet. This case has been studied
in a number of works~\cite{DKPr,prk,PratsKyurchev}, and the results of
the topological and metric analysis obtained therein have found
applications in various areas of mathematics, in particular, in the
theory of singularly continuous random variables of the
Jessen--Wintner type.

In the case of zero redundancy, the numbers having two representations,
namely:\\
    $$ \Delta^{A_s}_{c_1...c_{2m-1}c_{2m}(e_{s-1}e_0)}=
  \Delta^{A_s}_{c_1...c_{2m-1}[c_{2m}-d](e_0e_{s-1})},$$
  $$\Delta^{A_s}_{c_1...c_{2m-2}c_{2m-1}(e_{s-1}e_0)}=
  \Delta^{A_s}_{c_1...c_{2m-2}[c_{2m-1}+d](e_0e_{s-1})},$$
 are called \emph{$A_s$-binary numbers}, whereas those having a unique
representation are called \emph{$A_s$-unary numbers}.

 For $s=3$, the following equalities hold:
  \[\Delta^{A_3}_{c_1...c_{2k}e_2(e_2e_0)}=\Delta^{A_3}_{c_1...c_{2k}e_1(e_0e_2)},\;
  \Delta^{A_3}_{c_1...c_{2k}e_1(e_2e_0)}=\Delta^{A_3}_{c_1...c_{2k}e_0(e_0e_2)}.\]
    \[\Delta^{A_3}_{c_1...c_{2k-1}e_2(e_0e_2)}=\Delta^{A_3}_{c_1...c_{2k-1}e_1(e_2e_0)},\;
  \Delta^{A_3}_{c_1...c_{2k-1}e_1(e_0e_2)}=\Delta^{A_3}_{c_1...c_{2k-1}e_0(e_2e_0)}.\]
The set of all $A_s$-binary numbers is countable and dense in the
interval $[d_0,d_1]$.

\section{Applications of $A_s$-Representations of Numbers}
 
\begin{lemma}\label{lemma6}
Suppose that the $A_s$-representation of numbers has zero redundancy,
i.e., the conditions of Theorem~\ref{repreth} are satisfied. Then, for
any interval $u\subset[d_0,d_1]$, there are at most
$2(s^v+s-1)$ $A_s$-cylinders that cover $u$ and have length not exceeding
the length of $u$, where $v$ is the smallest positive integer satisfying
$C_2^v\leq C_1$, and $C_1,C_2$ are the constants whose existence is
asserted in Theorem~\ref{th2}.
\end{lemma}
\begin{proof}
Let $n+1$ be a positive integer such that $u$ contains an $A_{s}$-cylinder $\Delta^{A_{s}}_{c_1...c_{n+1}}$ of rank $n+1$, but contains no
$A_{s}$-cylinder of rank $n$. Clearly, it is possible that $u=\Delta^{A_{s}}_{c_1...c_{n+1}}$. If this is not the case, then there exist two $A_s$-cylinders
$w_0$ and $w_1$ of rank $n$ such that $u\subset w_0\cup w_1$, moreover $\max w_0=\min w_1$.

Let $u_0=u\cap w_0=[\min u;\max w_0]$, $u_1=u\cap w_1=[\max w_0;\max u]$. Without loss of generality, assume that $\Delta^{A_{s}}_{c_1...c_{n+1}} \subset u_0$. Then  $w_0=\Delta^{A_{s}}_{c_1...c_n}$, and $u_0$ can contain at most $s-1$ cylinders of rank $n+1$, since  $u_0\neq\Delta^{A_{s}}_{c_1...c_n}$. All these cylinders have length
not exceeding $|u_0|$ and, consequently, not exceeding $|u|$.

 These cylinders may cover $u_0$, but this is not guaranteed. If they do
not, there exists an $A_s$-cylinder $\Delta^{A_{s}}_{c_1...c_n j}$, which, together with the cylinders already considered, covers $u_0$.
If its length exceeds $|u_0|$, consider all cylinders of rank $n+v$
contained in $\Delta^{A_{s}}_{c_1...c_n j}$. Their number is $s^v$. Each of them has length not exceeding $|u_0|$.
Indeed, by Theorem~\ref{th2}, $$C_1|\Delta^{A_{s}}_{\alpha_1...\alpha_n}|\leq |\Delta^{A_{s}}_{\alpha_1...\alpha_n\alpha_{n+1}}|\leq C_2|\Delta^{A_{s}}_{\alpha_1...\alpha_n}|,$$ so, \[|\Delta^{A_{s}}_{c_1...c_nj\gamma_1...\gamma_{v-1}}|\leq
C_2^v|\Delta^{A_{s}}_{c_1...c_n}|\leq
C_1|\Delta^{A_{s}}_{c_1...c_n}|\leq|\Delta^{A_{s}}_{c_1...c_nc_{n+1}}|\leq |u_0|. \]
Thus, the cylinders of rank $n+1$ together with the cylinders of rank
$n+v$ form a cover of $u_0$ by cylinders whose lengths do not exceed
$|u_0|$. Hence, at most $s^v+s-1$ $A_s$-cylinders are sufficient to cover $u_0$.

Now consider the interval $u_1$, which is contained in $w_1=\Delta^{A_{s}}_{\beta_1...\beta_n}$. There is no guarantee that $u_1$ contains a cylinder of rank $n+1$.
However, by definition, there exists a smallest $k$ such that $\Delta^{A_{s}}_{\beta_1...\beta_n\underbrace{0...0}_k}\subset u_1$. Then, by the minimality of $k$, the interval $u_1$ can contain at most
$s-1$ cylinders of rank $n+k$. Their union may coincide with $u_1$, but
this is not guaranteed. If it does not, consider an
$A_s$-cylinder $\Delta^{A_{s}}_{\beta_1...\beta_n\gamma_1...\gamma_k}$, which is not contained in $u_1$ but has a nonempty intersection with
$u_1$. If its length exceeds $|u_1|$, consider all cylinders of rank
$n+k+v$ contained in $\Delta^{A_{s}}_{\beta_1...\beta_n\gamma_1...\gamma_k}$. Their number is $s^v$. Each of these cylinders has length not exceeding
$|u_1|$, which follows from the fact that the basic metric ratio for the
$A_s$-representation is bounded away from both zero and one. Hence,
at most $s^v+s-1$ $A_s$-cylinders are sufficient to cover $u_1$ as well. Since
$u=u_0\cup u_1,$
at most
$2(s^v+s-1)$
$A_s$-cylinders are sufficient to cover $u$. 
\end{proof}
\begin{theorem}
In the computation of the Hausdorff--Besicovitch dimension of Borel
subsets of the interval $[d_0,d_1]$, it suffices to consider coverings
by $A_s$-cylinders.
\end{theorem}
\begin{proof}
Let us recall the definition of the $\alpha$-dimensional Hausdorff measure and the Hausdorff--Besicovitch dimension of subsets of $R^1$. Let $\alpha$ be a positive real number. The $\alpha$-dimensional Hausdorff measure of a set $B\subset R^1$ is defined as the limit
\[
H^{\alpha}(B)=\lim\limits_{\varepsilon\to 0}m_{\varepsilon}^{\alpha}(B),
\qquad
m_{\varepsilon}^{\alpha}(B)=
\inf\limits_{d(B_k)\leq\varepsilon}
\{
\sum\limits_{k}|B_k|^{\alpha}:
\bigcup\limits_{k}B_k\supset B
\},
\]
where the infimum is taken over all coverings $\bigcup_k B_k$ of the set $B$ by intervals $B_k$ whose lengths $|B_k|$ do not exceed $\varepsilon$.

The Hausdorff--Besicovitch dimension $\alpha_0(B)$ of the set $B$ is defined by
\[
\alpha_0(B)=
\inf\{\alpha:H^{\alpha}(B)=0\}
=
\sup\{\alpha:H^{\alpha}(B)=\infty\}.
\]
 
Let $W$ denote the family of all $A_s$-cylinders, and let $W_n$ denote
the family of all $A_s$-cylinders of rank $n$.

Let $E$ be an arbitrary Borel subset of the interval $[d_0,d_1]$,
$\alpha>0$, $\varepsilon>0$, and let $U$ denote the family of all
intervals contained in $[d_0,d_1]$. Consider separately $\varepsilon$-coverings of $E$ by intervals and by
$A_s$-cylinders, together with the corresponding $\alpha$-contents,
defined by
 $$m_{\varepsilon}^{\alpha}(E)=\inf\limits_{|u_i|\leq\varepsilon}\{\sum\limits_{i}|u_i|^{\alpha}:
 E\subset \bigcup\limits_{i}u_i, u_i\in U, |u_i|\leq \varepsilon\},$$
 and
  $$l_{\varepsilon}^{\alpha}(E)=\inf\limits_{|w_i|\leq\varepsilon}\{\sum\limits_{i}|w_i|^{\alpha}:
 E\subset \bigcup\limits_{i}w_i, w_i\in W, |w_i|\leq \varepsilon\}.$$

We show that
\begin{equation}\label{mit}
  m_{\varepsilon}^{\alpha}(E)\leq l_{\varepsilon}^{\alpha}(E)\leq 2(s^v+s-1) m_{\varepsilon}^{\alpha}(E),
\end{equation}
where $v$ is the smallest positive integer satisfying
$C_2^v\leq C_1$. The left-hand inequality is obvious, since every $A_s$-cylinder is an
interval, and hence the class of coverings by $A_s$-cylinders is a
subclass of the class of all coverings by intervals.

The right-hand inequality follows from the lemma~\ref{lemma6}. Therefore, the Hausdorff $\alpha$-measure $H^\alpha(E)$ and $l^\alpha(E)=\lim\limits_{\varepsilon\to 0}l_{\varepsilon}^{\alpha}(E)$ simultaneously take the values $0$ and $\infty$ as $\alpha$ varies; that
is,
\[l^{\alpha}(E)=\begin{cases}
                  \infty & \mbox{при } \alpha<\alpha_0(E),\\
                  0 & \mbox{при } \alpha>\alpha_0(E),
                \end{cases}\]
where $\alpha_0(E)$ is the Hausdorff--Besicovitch dimension of $E$.
Hence, in computing the Hausdorff--Besicovitch dimension of subsets of
the interval $[d_0,d_1]$, it suffices to consider coverings by
$A_s$-cylinders.
\end{proof}
\begin{remark}
The preceding theorem is an analogue of Billingsley's theorem~\cite{Ergodic} for the representation of numbers in the $s$-adic numeral system.
\end{remark}

\end{document}